\documentclass[11pt,letterpaper]{amsart}   	
\usepackage{graphicx}				
\usepackage{caption}
\usepackage{subcaption}
\usepackage{amssymb}
\usepackage{amsthm}
\usepackage{csquotes}
\usepackage{amsmath}
\usepackage{mathtools}
\usepackage{float}
\usepackage{physics}
\usepackage{enumerate}
\usepackage{color}
\MakeOuterQuote{"}
\newtheorem*{theorem*}{Theorem}
\newtheorem{definition}{Definition}
\newtheorem{theorem}{Theorem}[section]
\newtheorem{proposition}[theorem]{Proposition}

\newtheorem{lemma}[theorem]{Lemma}

\newtheorem{remark}[theorem]{Remark}
\newtheorem{corollary}[theorem]{Corollary}

\newtheorem*{conjecture*}{Conjecture}
\title[Homothetic Self-Similar Solutions]{Homothetic Self-Similar Solutions to the Incompressible Navier-Stokes Equations}
\author{Tim Binz \& Matei P. Coiculescu}
\begin{document}
\begin{abstract}
We investigate homothetic forward self-similar solutions of the incompressible Navier-Stokes equations: the solutions $\overline{U}$ for which both $\overline{U}$ and $\beta \overline{U}$ are self-similar profiles for some nontrivial $\beta$. Homothetic solutions are, in addition, the only solutions for which a singular limit argument can be used to prove non-uniqueness of Leray-Hopf solutions along the lines of the Jia-\v{S}ver\'{a}k program. In three dimensions, and for sufficiently regular initial data, we prove a Liouville theorem that rules out the existence of non-trivial homothetic solutions. In two dimensions, our Liouville theorem reproves a result of Luo and Shvydkoy that the only decaying homothetic solution with continuous angular data is the Oseen vortex. In addition, we prove that the Euler equations linearized around the Oseen vortex are stable in some subspaces of $L^2(\mathbb{R}^2)$, for example in the space of $2$-fold rotationally symmetric functions. On the other hand, we also discover a homothetic solution for which an unstable approximate eigenvalue of the linearized Euler equations exists. 
\end{abstract}
\maketitle
\section{Introduction}
We consider the initial value problem of the unforced, incompressible Navier-Stokes equations in $\mathbb{R}^3$:
\begin{equation}
    \label{NSEQN}
    \left\{
\begin{aligned}
    \partial_t u -\nu \Delta u +u\cdot\nabla u +\nabla p &=0\\
    \textrm{div } u &=0\\
    u(t=0) &= u_0.
\end{aligned}
\right. 
\end{equation}
Since its value will not concern us, we set the viscosity parameter $\nu=1$. The existence of a global-in-time Leray-Hopf weak solution to Equation \eqref{NSEQN} is classical (see \cite{TS} for details), but the uniqueness question remains of contemporary interest.

One of the most promising avenues towards proving non-uniqueness of Leray-Hopf weak solutions is the following program proposed by Jia and \v{S}ver\'{a}k in \cite{JS1}. Given $\sigma>0$ and any $(-1)$-homogeneous weakly divergence free vector field $u_0(x) \in C^\alpha(\mathbb{R}^3\setminus\{0\})\subset L^{3,\infty}(\mathbb{R}^3)$, Jia and \v{S}ver\'{a}k proved in \cite{JS2} the existence of a spatially smooth, forward self-similar solution $U_\sigma(x,t)$ with initial data $\sigma u_0(x)$. For sufficiently small $\sigma>0$, the results of \cite{JS2} guarantee the uniqueness of the solution $U_\sigma$. 

Following \cite{JS1,JS2} we consider the self-similar ansatz appropriate to the parabolic scaling of the Navier-Stokes equations,
$$\xi = x t^{-1/2}, \quad \tau = \log(t), \quad u(x,t) = t^{-1/2}U(\xi, \tau),$$
and obtain for $U$ independent of $\tau$
\begin{equation}
    \label{SSNS}
    \left\{
    \begin{aligned}
        -\Delta U-\frac{U}{2} - \frac{\xi}{2}\cdot\nabla U +U\cdot\nabla U + \nabla P &=0\\
        \textrm{div } U &=0 .
    \end{aligned}
    \right. 
\end{equation}
The initial data $u_0$ enters as a boundary condition at infinity given by
\begin{equation}
        U(\xi) = u_0(\xi) +o\left(\frac{1}{|\xi|}\right), \quad |\xi|\to \infty.
        \label{eq:isotropic decay}
\end{equation}
A forward self-similar solution $\bar{u}_\sigma$ corresponds to a self-similar profile $\overline{U}_\sigma$, i.e. a steady solution of \eqref{SSNS}.
Consider the linearization of Equation \eqref{SSNS} around the self-similar profile $\overline{U}_\sigma$:
$$\mathcal{L}_\sigma U :=  \Delta U + \frac{U}{2} + \frac{\xi}{2}\cdot\nabla U -\overline{U}_\sigma\cdot\nabla U - U\cdot\nabla \overline{U}_\sigma- \nabla P.$$
Jia and \v{S}ver\'{a}k prove in \cite{JS1} that under some spectral conditions on $\mathcal{L}_\sigma$, the non-uniqueness of Leray-Hopf weak solutions to Equation \eqref{NSEQN} is guaranteed. The general idea is that as $\sigma$ increases in magnitude, the spectrum of $\mathcal{L}_\sigma$ will eventually cross the imaginary axis, leading to a bifurcation and the existence of another scale-invariant solution. Then, if the data $u_0$ decays like $1/|x|$, a truncation argument can be used to prove the non-uniqueness of Leray-Hopf solutions. We remark that Hou, Wang, and Yang have announced in \cite{CLAIM} a computer assisted completion of the Jia-\v{S}ver\'{a}k program. Very recently, Ionescu, Jia, and Palasek \cite{IJP} have provided strong numerical evidence of the existence of an unstable axisymmetric self-similar solution, which would complete the Jia-\v{S}ver\'{a}k program even in the axi-symmetric setting. 

For more singular data than the class considered in \cite{JS2} we use a distributional notion of boundary-value at infinity. Whenever
$u_0$ is specified, the assertion that $U$ has boundary value $u_0$
means
\begin{equation}\label{eq:blow-down limit}
R U(R\,\cdot)\longrightarrow u_0
\quad\hbox{in }\mathcal S'(\mathbb R^d),\qquad R\to\infty.
\end{equation}
For self-similar profiles with locally integrable angular data,
we state a stronger local convergence wherever it is needed.
We call a profile {\textit{isotropically decaying}} when
$|U(x)|\leq C(1+|x|)^{-1}$ for all $x$ for some constant $C$. We also consider {\textit{anisotropic profiles}} (i.e. profiles that are not isotropic)
which need not satisfy this bound. The heuristic to keep in mind is that a more regular data $u_0$ leads to better decay for the profile $U$. For the usual boundary conditions, data H\"{o}lder continuous away from the origin leads to isotropically decaying profiles, while in this work we construct smooth solutions from more singular $BMO^{-1}$ data that only decay in one direction.

Albritton, Bru\'{e}, and Colombo were able to carry out the Jia-\v{S}ver\'{a}k program in \cite{ABC} by introducing a force. The force enabled them to consider the linearization $\mathcal{L}_\sigma$ around an arbitrary, smooth, divergence free vector field $\sigma\cdot \overline{U}$ for any positive number $\sigma > 0$, not necessarily around an exact self-similar profile. Then, by proving the existence of an unstable eigenvalue to the linearization of the Euler equations (for some particular choice of $\overline{U}$):
\begin{equation*}
L_\phi U := -\overline{U} \cdot\nabla U - U\cdot\nabla \overline{U} -\nabla P,
\end{equation*}
they were able to prove the existence of an unstable eigenvalue for the linearization of the self-similar Navier-Stokes equations 
\begin{equation*}
\mathcal{L}_\sigma U = \frac{U}{2} + \frac{x}{2}\cdot\nabla U +\Delta U +\sigma L_\phi U
\end{equation*}
around $\sigma \bar{U}$ for some $\sigma > 0$. Their proof relies on a {\textit{singular limit argument}}: consider the operator
\begin{equation*}
    \tilde{\mathcal{L}}_\sigma U = \sigma^{-1} \mathcal{L} U = \sigma^{-1}\left(\frac{U}{2} + \frac{x}{2}\cdot\nabla U +\Delta U\right) + L_\phi U .
\end{equation*}
Albritton, Bru\'{e}, and Colombo prove that 
\begin{equation*}
    \tilde{\mathcal{L}}_\sigma \to L_\phi \quad  \text{ as } \quad \sigma \to + \infty
\end{equation*}
in the strong resolvent sense uniformly on compact sets in $\{z : \textrm{Re}(z) >0\}\cap \rho(L_\phi) $, which implies the convergence of isolated eigenvalues. This allows the authors to conclude the existence of an unstable eigenvalue for $\tilde{\mathcal{L}}_\sigma$, hence for $\mathcal{L}_\sigma$, when $\sigma > 0$ is sufficiently large. Our work is motivated by the following question: can the singular limit argument of Albritton, Bru\'{e}, and Colombo work for the unforced Navier-Stokes equations? 

One thing that is necessary for the argument to work is that $\bar{U}$ and $\beta \bar{U}$ are self-similar solutions for some non-trivial $\beta$, which, we shall see, is far from being true for a generic self-similar solution. 
In addition, the singular limit argument from \cite{ABC} relies heavily on the profile $\overline{U}$ having compact support, which is not feasible in the unforced case, since as we can see from the boundary conditions, any compactly supported self-similar solution corresponds to zero initial data at physical time $t=0$.

Once again, our key observation is that the singular limit argument of \cite{ABC} can only be carried through for a specific class of self-similar solutions:
\begin{definition}
  A self-similar solution $\bar{U}$ is called a {\textbf{homothetic}} self-similar solution if and only if there exists a real number $\beta$ with $\beta\neq 1,0$ such that $\beta \bar{U}$ is also a self-similar solution.
\end{definition}

We shall often simply call such solutions homothetic solutions. As we shall see, for the Navier-Stokes equations, once $\bar{U}$ and $\beta \bar{U}$ are solutions for some $\beta \neq 1,0$, $\beta U$ is a self-similar solution {\textit{for all}} $\beta\in \mathbb{R}$. 
It is worth pointing out that homothetic solutions are uniquely determined by their boundary data: for a given (-1)-homogeneous divergence free $C^\alpha$-vector-field $u_0$ we shall show that there exists at most one homothetic self-similar solution whose blow-down limit is $u_0$. Therefore, non-uniqueness from such data necessarily leads to another profile breaking the homothetic property.


Our main result rules out the existence of sufficiently regular homothetic solutions. Indeed, we show the following Liouville type theorem:

\begin{theorem}
    \label{LIOU}

    Suppose $\bar{U}$ is a homothetic forward self-similar solution to the three-dimensional Navier-Stokes equations lying in the Morrey space $\mathcal{M}^3_q(\mathbb{R}^3)$ for some $2<q\leq 3$ and corresponding to initial data
    $$ u_0 = \frac{\vec{v}+f\vec{n}}{|x|}\in \mathcal{M}^3_q(\mathbb{R}^3),$$
    where $\vec{n}$ is the unit normal to the sphere and $\vec{v}$ is a tangent vector field to the sphere. Suppose either
    \begin{enumerate}
        \item $\vec{v}\in W^{1,4/3}(\mathbb{S}^2)$ and $f\in L^4(\mathbb{S}^2)$
        \item $\vec{v},f \in C^\alpha(\mathbb{S}^2)$, for some $\alpha>1/2$
    \end{enumerate} 
    then $f=0$ and the head pressure $|\vec{v}|^2+2p=0$ as well. In case (2) we get that $\vec{v}=u_0=\bar{U}=0$ with no additional assumptions. For case (1), if we assume that the set $\{\omega\in \mathbb{S}^2 : \vec{v}(\omega)=0\}$ has measure zero or is closed with $C^2$ boundary, then $\vec{v}=u_0=\bar{U}=0$.
\end{theorem}

For the reader's benefit, we recall the definition of the Morrey spaces $\mathcal{M}^p_q(\mathbb{R}^3)$ where $1\leq q \leq p$:
$$f\in \mathcal{M}^p_q(\mathbb{R}^3) \iff \exists C>0, \forall x\in \mathbb{R}^3, \forall R>0, \quad \int_{B_R(x)} |f(y)|^q dy \leq CR^{3-\tfrac{3q}{p}}.$$
Evidently, since $u_0\in \mathcal{M}^3_q(\mathbb{R}^3)$, we have $\vec{v}, f\in L^q(\mathbb{S}^2)$, so there is a gap between a reasonable class of data ($L^q$ with $2<q\leq 3$) and the level at which the Liouville theorem can be proven ($L^4$). Indeed, the existence of a unique mild solution to the Navier-Stokes equations has even been shown for small data $u_0\in \mathcal{M}^3_q(\mathbb{R}^3)$ for any $q>1$. 

The main idea behind our proof of Theorem \ref{LIOU} is that the blow-down limit of a homothetic solution in three-dimensions lying in the Morrey space $\mathcal{M}^3_q$ with $q>2$ is always a $(-1)$-homogeneous weak stationary solution of the three-dimensional incompressible Euler equations. The proof then continues by showing the nonexistence of nontrivial $(-1)$-homogeneous Euler solutions in the regularity classes dealt with in Theorem \ref{LIOU}. More precisely, we prove the following Liouville theorem for the stationary Euler equations:

\begin{theorem}\label{EULER3D}
    Suppose $u_0$ is a $(-1)$-homogeneous weak stationary solution of the three-dimensional incompressible Euler equations,
    \[
    \operatorname{div}(u_0\otimes u_0)+\nabla P_0=0,
    \qquad \operatorname{div}u_0=0
    \quad\hbox{in }\mathcal D'(\mathbb R^3),
    \]
    of the form
    \[
    u_0(r\omega)=\frac{\vec{v}(\omega)+f(\omega)\vec{n}(\omega)}{r},
    \qquad P_0(r\omega)=\frac{p(\omega)}{r^2},
    \]
    where $\vec{n}$ is the unit normal to the sphere and $\vec{v}$ is a tangent vector field to the sphere. Suppose either
    \begin{enumerate}
        \item $\vec{v}\in W^{1,4/3}(\mathbb{S}^2)$ and $f\in L^4(\mathbb{S}^2)$
        \item $\vec{v},f \in C^\alpha(\mathbb{S}^2)$, for some $\alpha>1/2$
    \end{enumerate}
    then $f=0$ and the head pressure $|\vec{v}|^2+2p=0$ as well. In case (2) we get that $\vec{v}=u_0=0$ with no additional assumptions. For case (1), if we assume that the set $\{\omega\in \mathbb{S}^2 : \vec{v}(\omega)=0\}$ has measure zero or is closed with $C^2$ boundary, then $\vec{v}=u_0=0$.
\end{theorem}

Our Liouville theorem for the steady Euler equations is a generalization of the result of Shvydkoy \cite{Sh}, cf. our Proposition \ref{prop:sh}. We also mention that three-dimensional homogeneous Euler flows have been studied in \cite{Sh} and \cite{ABE}, and two-dimensional Euler flows have been studied in \cite{LS} and \cite{CC}.

Let us point out that while for the Jia-\v{S}ver\'{a}k program the boundary data is allowed to be in $C^\alpha(\mathbb{S}^2)$ for any $\alpha > 0$, our Theorem \ref{LIOU} only presently rules out the existence of homothetic solutions with boundary data in $C^\alpha(\mathbb{S}^2)$ for $\alpha > \frac{1}{2}$.
Therefore, one may now wonder whether our Theorem \ref{LIOU} can be generalized to all $\alpha >0$. The bottleneck in our proof is the Liouville theorem for $(-1)$-homogeneous solutions to the steady Euler equations. It is unlikely that this result can be generalized to arbitrary $(-1)$-homogeneous $C^\alpha$-solutions below the Onsager threshold $\alpha=\frac{1}{3}$. Instead, we expect the existence of wild solutions arising from convex integration in this regime. It is unknown to us whether such irregular boundary data can yield isotropically decaying homothetic self-similar solutions or not. We also remark that the Morrey space condition, at the very least, appears sharp. In this work we construct explicit nontrivial examples of homothetic solutions lying in $\mathcal{M}^3_q(\mathbb{R}^3)$ for all $q<2$. The question of uniqueness for these exact profiles remains open.

While the idea of Albritton, Bru\'{e}, and Colombo only directly applies to homothetic solutions, a more general strategy applies to rescaled self-similar profiles that do not themselves have to solve the self-similar Navier-Stokes equations. For a more general family $U_\beta$ with boundary value $\beta u_0$,
let $V_\beta=U_\beta/\beta$ and rescale the pressure so that
\[
-\beta^{-1}\left(\Delta V_\beta+
\tfrac12(V_\beta+x\cdot\nabla V_\beta)\right)
+(V_\beta\cdot\nabla)V_\beta+\nabla Q_\beta=0.
\]
A formal limit of $V_\beta$ to a solution $V_\infty$ of the steady Euler equations would require compactness sufficient to pass
to the nonlinear term in the weak formulation. To guarantee that the limit $V_\infty$ has the same blowdown limit boundary value at infinity, we also would require a uniform in $\beta$ blowdown limit for the family $V_\beta$.
Note that although the equation satisfied by $V_\beta$ changes with $\beta$, the boundary data, assuming a uniform in $\beta$ blowdown limit, is constant in $\beta$. Inspired by the work in \cite{JS1} and \cite{ABC}, we have investigated when a sufficient condition for non-uniqueness of Leray-Hopf solutions of the Navier-Stokes equations can be found in terms of the linearized Euler equations in the unforced case. In forthcoming work, we prove a conditional non-uniqueness theorem for Leray-Hopf solutions along the lines of this {\textit{inviscid-limit argument}} (which we may contrast with the previous singular-limit argument as not requiring the rescaled profiles to solve the same PDE).

On the other hand, we are also able to rule out this inviscid-limit scenario for sufficiently regular data. Due to the scaling invariance of the Euler equations, the vector field $M V_\infty(M \cdot\xi)$ also solves the steady Euler equations with the same data. Under mild regularity assumptions on $V_\infty$ (e.g. it lies in the same Morrey space as in the statement of Theorem \ref{LIOU}), and assuming strong local $L^2$ convergence of its blow-downs, by following the proof of our Proposition \ref{NSTOEULER}, we can obtain that the $(-1)$-homogeneous data $u_0$ must be a steady-state of the Euler equations. 
Thus, our Theorem \ref{EULER3D} applies, confirming the triviality of $u_0$. 
Consequently, the inviscid singular limit argument is ruled out for (sufficiently regular) arbitrary nontrivial data.

This shows that the curve of self-similar solutions $U_\beta$ constructed by Jia and \v{S}ver\'{a}k \cite{JS2} cannot admit such an inviscid limit with nontrivial boundary data. This does not by itself imply stability for sufficiently large $\beta$ or restrict the desired instability to an intermediate regime of $\beta$.

For anisotropically decaying homothetic solutions, in contrast with the work above, we are able to find examples with boundary data below the regularity threshold of Theorem \ref{LIOU}, see Section~\ref{sec:examples}. This confirms that the regularity conditions in Theorem \ref{LIOU} cannot be removed in general. 
Indeed, a few such homothetic solutions have been described previously in the work \cite{O} of Okamoto, but we believe some of our examples are new self-similar solutions of the three-dimensional Navier-Stokes equations. The construction of these solutions is, surprisingly, more geometric than analytic in nature. Moreover, we shall prove that these solutions sometimes admit instabilities of the linearized Euler equations in the form of unstable normal modes, which can be used to make approximate eigenvalues via a wave-packet construction. As an example, we consider the shear flow described by Okamoto in \cite{O}:
$$\overline{U}_{1,\beta}(\xi_1,\xi_2,\xi_3)=\beta\left(\frac{\sqrt{\pi}}{2} e^{-\frac{\xi_2^2}{4}} \text{erfi}\left(\frac{\xi_2}{2}\right),0,0\right).$$
The vector field $\overline{U}_{1,\beta}$ is an anisotropically decaying homothetic self-similar solution to the Navier-Stokes equations from initial data:
$$\overline{u}_{1,\beta}(x,y,z) = \textrm{P.V. } \left( \frac{\beta}{y},0,0\right).$$
Note that this solution is outside the regularity range of our previously discussed Liouville theorems, indeed the data lies in $BMO^{-1}$, so the unique existence of a mild solution follows from the work done by Koch and Tataru in \cite{KT}, as long as $\beta$ is sufficiently small. The second author and Palasek have proven in \cite{CP} that non-uniqueness can occur for data in the class $BMO^{-1}$, but the uniqueness question for any particular initial data, such as $\overline{u}_{1,\beta}$ with large $\beta$, remains open. Interestingly, we are able to prove the existence of an unstable element of the approximate point spectrum:

\begin{theorem}\label{UNSTABLE}
Let $\beta\neq 0$ and let $\overline{U}_{1,\beta}$ be the Okamoto shear flow described above. Let $L_\beta$ be the linearization of the Euler equations around $\overline{U}_{1,\beta}$ in velocity form on
$L^2_\sigma(\mathbb R^3;\mathbb C^3)$ with dense maximal domain. Then $L_\beta$ has a positive real
point in its approximate spectrum.
\end{theorem}


Although one may think that we should be able to improve Theorem \ref{UNSTABLE} to a non-uniqueness result, there are clear technical obstacles to doing so. First, the singular limit argument that allows one to conclude the existence of an unstable eigenvalue of $\mathcal{L}_\sigma$ crucially depends on the existence of a discrete eigenvalue (rather than approximate point spectrum) to allow the instability to pass through the singular limit. Second, the lack of decay of the background profile makes several parts of the argument, for instance the truncation of the data, unfeasible. Third, attempts to use custom anisotropic function spaces for the singular limit argument are usually defeated by the drift operator $x\cdot\nabla $, which behaves poorly on such spaces. Lastly, the existence of the unstable spectrum for all $\beta\neq 0$ suggests a dynamical behavior different than the discrete eigenvalue crossing bifurcation used in the program of \cite{JS1}.

%

\medskip 


Let us now discuss the analogous situation in two dimensions. We emphasize that, in contrast to the three-dimensional situation, the existence of an unstable self-similar profile for the two-dimensional self-similar Navier--Stokes equations would not imply non-uniqueness of Leray--Hopf solutions. Rather, the Jia--\v{S}ver\'{a}k program would imply non-uniqueness in  $\mathrm{L}^p_t \mathrm{L}^q_x$ 
with $\frac{2}{p}+\frac{2}{q} > 1$, which would confirm that Serrin's condition is sharp. For a detailed discussion of the Jia--\v{S}ver\'{a}k program in dimension two we refer to the recent work of Albritton et al. \cite{AGKR}. 
Again, we rule out the singular limit scenario by proving the following results.
The first one shows that the Oseen vortex is the only isotropically decaying homothetic forward self-similar solution to the two-dimensional Navier-Stokes equations with continuous angular data, which is a consequence of the following Liouville theorem for the stationary Euler equations:

\begin{theorem}\label{EULER2D}
Let $b\in C(\mathbb S^1;\mathbb R^2)$ and set
$u_0(r\omega)=r^{-1}b(\omega)$ for $r>0$.
Suppose that $\operatorname{div}u_0=0$ in
$\mathcal D'(\mathbb R^2)$ and that, for some scalar distribution
$P_0$ on $\mathbb R^2\setminus\{0\}$,
\[
\operatorname{div}(u_0\otimes u_0)+\nabla P_0=0
\quad\hbox{in }\mathcal D'(\mathbb R^2\setminus\{0\}).
\]
Then there is $A\in\mathbb R$ such that
\[
b(\omega)=A\omega^\perp,
\qquad
u_0(r\omega)=\frac{A}{r}\omega^\perp.
\]
\end{theorem}

For homothetic solutions, Theorem \ref{EULER2D} and Lemma \ref{lem:heat-converse} give the following:

\begin{theorem}\label{2DTHM}
Let $b\in C(\mathbb S^1;\mathbb R^2)$ and set
$u_0(r\omega)=r^{-1}b(\omega)$ for $r>0$.
Suppose that $\operatorname{div}u_0=0$ in
$\mathcal D'(\mathbb R^2)$ and that, for some scalar distribution
$P_0$ on $\mathbb R^2\setminus\{0\}$,
\[
\operatorname{div}(u_0\otimes u_0)+\nabla P_0=0
\quad\hbox{in }\mathcal D'(\mathbb R^2\setminus\{0\}).
\]
Let $U$ be a homothetic solution with
$RU(R\,\cdot)\to u_0$ in $\mathcal S'(\mathbb R^2)$ as $R\to\infty$.
Then there is $A\in\mathbb R$ such that
\[
b(\omega)=A\omega^\perp,
\qquad
\bar U(r,\omega)=A\frac{1-e^{-r^2/4}}r\omega^\perp.
\]
Moreover, this profile's blow-down limit converges to $u_0$ in
$L^1_{\mathrm{loc}}(\mathbb R^2)$.
\end{theorem}
The proof of Theorem \ref{EULER2D} more or less follows the outline of the analogous argument in \cite{LS}. 
The long-time stability of the Oseen vortex for the two-dimensional self-similar Navier--Stokes equations was established by \cite{GW}; however, their result does not directly apply to the Jia-\v{S}ver\'{a}k program since their analysis is carried out in weighted function spaces. We prove the following stability result which rules out the singular-limit scenario in the two-dimensional case.

\begin{theorem}
\label{OSEENSTABLE}
 The linearization of the Euler equations in vorticity form around the Oseen vortex generates a strongly-continuous semigroup on the space $\mathrm{L}^2_2(\mathbb{R}^2)$ and 
 $$s(L_\phi, L^2_2(\mathbb{R}^2)) = \omega_0(L_\phi,L^2_2(\mathbb{R}^2))=0.$$
\end{theorem}
Here $L^2_2(\mathbb R^2)$ denotes the complex space
$\{g\in L^2(\mathbb R^2;\mathbb C):g(-x)=g(x)\}$ of
square-integrable two-fold rotationally symmetric functions.
Here $s(L,X)$ denotes the spectral bound of a closed, densely defined operator 
$L \colon D(L) \subset X \to X$ and $\omega_0(L,X)$ denotes the growth bound of the strongly continuous semigroup generated by $L$.

The argument used in the proof of Theorem \ref{EULER2D} also applies more generally to the inviscid limit, as previously discussed for the three-dimensional case. As before suppose $U_\beta$ is a solution to Equation \eqref{SSNS} with boundary data $\beta u_0(\xi)$ at spatial infinity for some $\beta>0$. Let $V_\beta(\xi) = U_\beta(\xi)/\beta$. If we take the formal limit as $\beta\to \infty$, then $V_\infty$, the limit of $V_\beta$ as $\beta\to \infty$, satisfies the steady Euler equations with $(-1)$-homogeneous data $u_0$. Due to the scaling invariance of the Euler equations, the vector field $M V_\infty(M \cdot\xi)$ also solves the steady Euler equations with the same data. From here, we can no longer argue completely analogously to the three-dimensional case, because $(-1)$-homogeneous functions are not locally square-integrable and therefore cannot be weak solutions of the Euler equations on $\mathbb{R}^2$ (Proposition \ref{NSTOEULER} does not directly apply). However, a similar argument under strong local $L^2$ convergence of the blow-down limits away from the origin can get that the $(-1)$-homogeneous data $u_0$ must be a steady-state of the Euler equations on $\mathbb{R}^2\setminus\{0\}$. Thus, our Theorem \ref{EULER2D} applies, which shows that the boundary data must be the point vortex $u_0(r\omega)=A\omega^\perp/r$. The corresponding homothetic self-similar solution is the Oseen vortex by Theorem \ref{2DTHM}.
Now, due to the stability of the Oseen vortex as proven in Theorem \ref{OSEENSTABLE}, we can say that we rule out the homothetic singular-limit argument of \cite{ABC} or \cite{AGKR} for non-uniqueness from any continuous angular data in two-dimensions within the two-fold rotationally symmetric space considered in Theorem \ref{OSEENSTABLE}. The Liouville theorem alone does not identify a general inviscid limit with the Oseen vortex or imply high-Reynolds stability. We contrast this stability result with the numerically discovered instability at intermediate Reynolds number, see \cite{AGKR}.

In the course of proving the stability of the Oseen vortex, we rediscovered the following general stability criterion for the Euler equations on the whole plane, also found in a separate context by Gallay and Smets in \cite{GS}.

\begin{theorem}
Let $\overline u=u(r)e_\theta$ be a smooth real vector field on
$\mathbb R^2$. Set $\Omega=u/r$ and $\overline\omega=u'+u/r$,
and suppose that $\Omega$ is bounded and
$\overline\omega'(r)=O(r^{-3})$ as $r\to\infty$.
If either $\overline\omega'(r)>0$ for every $r>0$ or
$\overline\omega'(r)<0$ for every $r>0$, then the spectrum of
the Euler equations in vorticity form linearized around $\overline{\omega}$ and densely defined on its maximal domain for the space $L^2_2(\mathbb R^2)$
is contained in the imaginary axis.
\end{theorem}

\subsection{Announcement of Stronger Rigidity Results}

In future work, we shall present proofs of the following improvements to the rigidity results in this paper. One important reason for separating the proof of the announced results from those within this paper is that the former are proven with LLM-assistance.

First, we shall prove
\begin{theorem}
    Let $U$ be a homothetic profile on $\mathbb{R}^3$ with data satisfying $u_0\in L^q(\mathbb{S}^2,\mathbb{R}^3)$ for any $q>2$. Then $U=0$.
\end{theorem}
This sharp result (we construct explicit examples in our work here when $q<2$) shows that the only plausible case remaining is if $u_0\in L^q(\mathbb{S}^2,\mathbb{R}^3)$ for $q\leq 2$. Note that existence of smooth mild solutions in the Morrey space $\mathcal{M}^3_q$ is guaranteed for all $1<q\leq 2$, so a non-uniqueness program may still be a possibility in these spaces. On the other hand, we have shown rigidity up to the space $L^2(\mathbb{S}^2)$ which is the borderline space for the truncation argument in \cite{JS1} to have any chance at working. Moreover, the classical existence theory of \cite{JS2} is completely excluded (since it deals with continuous spherical data) and we expect homothetic solutions from rough $L^2(\mathbb{S}^2)$ data, if they exist, to decay anisotropically. The proof of the theorem above no longer relies on a Liouville theorem for the stationary Euler equations.
It therefore no longer applies to the inviscid limit. In order to rule out the inviscid limit we still rely on the strategy of this paper. Here, the result is improved by sharpening the Liouville result for the (-1)-homogeneous stationary Euler equations. Indeed Theorem~\ref{EULER3D} can be improved to the following result:
\begin{theorem}
Suppose $u_0$ is a $(-1)$-homogeneous solution of the steady Euler equations on $\mathbb{R}^3$ such that
    $$ u_0 = \frac{\vec{v}+f\vec{n}}{|x|},$$
    where $\vec{n}$ is the unit normal to the sphere and $\vec{v}$ is a tangent vector field to the sphere. If any of the following conditions hold:
\begin{itemize}
    \item $v,f\in C^\alpha(\mathbb{S}^2)$ with $\alpha>1/3$
    \item $v\in c^{1/3}(\mathbb{S}^2)$ and $ f\in C^{1/3}(\mathbb{S}^2)$
    \item $v,f \in L^2(\mathbb{S}^2)$ and axisymmetric
    \item $f\in L^\infty(\mathbb{S}^2)$ and $v\in L^\infty(\mathbb{S}^2)\cap BV(\mathbb{S}^2)$
\end{itemize}
then $u_0=0$. 
\end{theorem}
This result shows that the $(-1)$-homogeneous solutions of the Euler equations must have regularity below the Onsager threshold, cannot be axisymmetric, and cannot be of bounded variation. Together, this result suggests a complicated geometric structure for these solutions, if such objects exist.

Following the strategy from this paper this implies for the inviscid limit:
\begin{corollary}
Suppose $U_\beta$ is a family of solutions to Equation \eqref{SSNS} on $\mathbb{R}^3$ with boundary data $\beta u_0$, where $u_0=(\vec{v}+f\vec{n})/|x|$ and $\vec{v},f$ satisfy any one of the regularity conditions in the preceding theorem. Set $V_\beta=U_\beta/\beta$. Suppose that, for some sequence $\beta_j\to\infty$,
    \[
    V_{\beta_j}\longrightarrow V_\infty
    \quad\hbox{strongly in }L^2_{\mathrm{loc}}(\mathbb{R}^3),
    \]
    and that the boundary data is attained uniformly along this sequence in the sense that, for every compact set $K\subset\mathbb{R}^3$,
    \[
    \lim_{R\to\infty}\sup_j
    \big\|R V_{\beta_j}(R\,\cdot)-u_0\big\|_{L^2(K)}=0.
    \]
    Then $u_0=0$. In particular, for nontrivial data satisfying any of these regularity conditions, no such inviscid limit exists.
\end{corollary}

Finally, in two dimensions we can show that any homothetic profile is either an Oseen vortex or an Okamoto shear flow, which sharpens the result from this work. This completely characterizes the homothetic profiles in two dimensions regardless of decay.

\subsection*{Acknowledgments}
	We thank Princeton University for its support. 
    MPC also thanks the Simons Foundation and New York University for their support.
	TB acknowledges the support of the DFG Walter-Benjamin Fellowship no. 538212014. We thank Dallas Albritton, Camillo De Lellis, Hao Jia, and Vlad Vicol for helpful discussions on this work. We used Codex to proofread and to audit the mathematics in this paper.

\section{Homothetic Self-Similar Solutions}

Let $U$ be a homothetic self-similar solution to the Navier-Stokes Equations. By definition there exists $\beta$ with $\beta\neq 1,0$ such that $U$ and $\beta U$ are both self-similar profiles; in particular, both solve Equation \eqref{SSNS}. Henceforth we shall always assume that any profile $U$ in $\mathbb{R}^d$ with $d\in \{2,3\}$ we analyze satisfies:
\begin{equation}
    \label{WEAKDECAY}
 U \in \mathcal{S}'(\mathbb{R}^d),\qquad \lim_{R\to\infty}RU(R\cdot) =u_0(\cdot) \textrm{ in } \mathcal{S}'(\mathbb{R}^d).
\end{equation}
This excludes pathological cases such as when harmonic gradients survive in the self-similar heat equation or the case of constant profiles. We thus work with any tempered distribution for which the natural blow-down limit exists. All the explicit examples we deal with here (such as the Okamoto shear flow) satisfy the condition in Equation \eqref{WEAKDECAY}.
Then we have the following short lemma:

\begin{lemma} $U$ is a homothetic self-similar solution to the Navier-Stokes equations if and only if there exists a scalar pressure $P$ such that
    \begin{equation}\begin{gathered}
        U\cdot\nabla U + \nabla P=0\\
        \frac{U}{2}+\frac{x}{2}\cdot\nabla U + \Delta U=0\\
        \textrm{div  } U =0.
    \end{gathered}\end{equation}
    Moreover, if $U$ is homothetic, then $\beta U$ is a solution for all $\beta \in \mathbb{R}$.
\end{lemma}
\begin{proof}
    By hypothesis we know that there exists $P$ and $P_\beta$ such that 
    $$-\frac{U}{2}-\frac{x}{2}\cdot\nabla U -\Delta U +U\cdot\nabla U + \nabla P=0$$
    and
    $$-\beta\frac{U}{2}-\beta\frac{x}{2}\cdot\nabla U -\beta\Delta U +\beta^2 U\cdot\nabla U + \nabla P_\beta=0.$$
Put $\mathcal A=\Delta+\tfrac12(1+x\cdot\nabla)$.
Subtracting $\beta^2$ times the first equation from the second gives
\[
 \mathcal A U=\nabla H,\qquad
 H=\frac{\beta^2P-P_\beta}{\beta^2-\beta},\qquad \Delta H=0.
\]
By our assumption in Equation \eqref{WEAKDECAY}, we have that $U$ is a tempered distribution, so $AU$ is a tempered harmonic
distribution and hence a polynomial, which we denote by $q$.
Set $U_R(x)=RU(Rx)$. Then
\[
Rq(R\cdot)
=R^{-2}\Delta U_R+\frac12(U_R+x\cdot\nabla U_R).
\]
By the assumption in Equation \eqref{WEAKDECAY}, the right-hand side converges in $\mathcal S'$
as $R\to\infty$. No nonzero polynomial $q$ can have
$Rq(R\cdot)$ converge in $\mathcal S'$. Thus $AU=0$, and the first equation
gives $U\cdot\nabla U+\nabla P=0$.
Conversely these two identities imply that every $\beta U$
is a profile, with pressure $\beta^2P$.
\end{proof}

Our first goal becomes one of constructing solutions of the Ornstein-Uhlenbeck-type equation (see \cite{MLP} and \cite{M} for work on similar systems)
\begin{equation}
    \label{OU}
    \begin{gathered}
    \frac{U}{2}+\frac{x}{2}\cdot\nabla U + \Delta U =0\\
    \textrm{div } u =0
    \end{gathered}
\end{equation}
from $(-1)$-homogeneous data $u_0(x)$ at infinity, in the sense that the following blow-down limit holds in $\mathcal S'$:
\begin{equation}
    \label{BLOWDOWN}
    \lim_{R\to \infty} RU(R\cdot x) = u_0(x).
\end{equation}

This corresponds to constructing forward self-similar solutions of the heat equation from $(-1)$-homogeneous data at initial time. We therefore consider the following data:
$$u_0(x)= |x|^{-1} \cdot v(\tfrac{x}{|x|}).$$

Our construction has two steps: from the data we first construct a distribution $a \colon \mathbb{S}^2 \to \mathbb{C}^3$ which captures the angular information of the solution $U$, then we construct $U$ from $a$. What permits this route is the radial symmetry of the Ornstein-Uhlenbeck operator. In our work below, the notation $n!!$ refers to the usual double factorial. The space $H^{-1/2}(\mathbb{S}^2,\mathbb{C}^3)$ is defined in terms of spherical harmonics $a_{kl}$ as the Hilbert space with the following norm
$$\|a\|_{H^{-1/2}(\mathbb{S}^2,\mathbb{C}^3)} := \sum_{k=0}^\infty \sum_{k=-l}^{k=l} (1+k(k+1))^{-1/2}|a_{kl}|^2.$$
We also recall that the Funk-Hecke formula (see \cite{Seeley1966}) for integration against a kernel on the sphere is
$$\int_{\mathbb{S}^2} K(\omega\cdot\Omega)Y_{kl}(\Omega)d\Omega = 2\pi \left(\int_{-1}^{1}K(t)P_k(t)dt\right) Y_{kl}(\omega),$$
where $K$ is some kernel function, $P_k$ is the Legendre polynomial at order $k$, and $Y_{kl}$ is the spherical harmonic at its respective order. 
\begin{lemma}[from $v$ to $a$]
\label{lemma:v to a}
Let $v\in L^2(\mathbb S^2;\mathbb R^3)$ and let $Y_{kl}$ be the
orthonormal spherical harmonics.
Define $a=Tv$ harmonic-wise by the multipliers
\[
T Y_{kl}=t_kY_{kl},\qquad
t_k=\begin{cases}
\displaystyle\frac{(-1)^{(k-1)/2}}{4\pi i}
\frac{k!!}{(k-1)!!},&k\text{ odd},\\[3pt]
\displaystyle\frac{(-1)^{k/2}}{2\pi^2}
\frac{2^k((k/2)!)^2}{k!},&k\text{ even}.
\end{cases}
\]
Then $a\in H^{-1/2}(\mathbb S^2;\mathbb C^3)$ and
$a(-\omega)=\overline{a(\omega)}$. With
$a_\pm(\omega)=(a(\omega)\pm a(-\omega))/2$, the identity
\[
|x|^{-1}v(x/|x|)
=\pi\langle a_+(\omega),\delta(\omega\cdot x)\rangle_\omega
+i\left\langle a_-(\omega),
\operatorname{PV}\frac1{\omega\cdot x}\right\rangle_\omega
\]
holds in $\mathcal S'(\mathbb R^3)$. These expressions mean
that the kernel is first tested in $x$ against a Schwartz
function, and only then paired with the distribution in
$\omega$.
\end{lemma}
\begin{proof}
The factorial ratios and Stirling's approximation give $|t_k|\leq C(1+k)^{1/2}$. Hence,
writing $v=\sum c_{kl}Y_{kl}$, we see
\[
\sum_{k,l}(1+k(k+1))^{-1/2}|t_kc_{kl}|^2
\leq C\sum_{k,l}|c_{kl}|^2,
\]
which shows that $a\in H^{-1/2}(\mathbb{S}^2)$.
The even multipliers are real and the odd multipliers are
purely imaginary, which proves the symmetry $a(-\omega)=\overline{a(\omega)}$ that guarantees that $v$ is real. Now we prove the distributional identity.

For a finite harmonic sum, we apply the Funk-Hecke formula to the even kernel $\delta_0(t)$ and to the odd kernel $\textrm{P.V.}(1/t)$. The following Legendre polynomial identities for $P_{2m}(0)$ and
$\int_{-1}^{1} P_{2m+1}(t)/t\,dt$ are, respectively, from
\cite[formula~(18.12.11)]{DLMF} with $x=0$ and
\cite[formula~(18.17.39)]{DLMF} with $z=0$, using parity:
\[
P_{2m}(0)=(-1)^m\frac{(2m)!}{2^{2m}(m!)^2},\qquad
\int_{-1}^1\frac{P_{2m+1}(s)}s\,ds
=2(-1)^m\frac{(2m)!!}{(2m+1)!!}.
\]
These identities show that the harmonic-wise multipliers are exactly $t_k^{-1}$.
This proves the claimed distributional identity for finite sums.
For Schwartz test functions in $x$, testing either kernel
produces a smooth function of $\omega$: locally rotate
$\omega$ to a fixed coordinate direction and differentiate
the rotated Schwartz test function. Thus both kernel expressions
depend continuously on $a$ as distributions. Partial sums of spherical harmonics converge to $v$ in $L^2$ and to $Tv$ in
$H^{-1/2}$. The homogeneous extensions from $\mathbb{S}^2$ to $\mathbb{R}^3$ converge in
$\mathcal S'(\mathbb R^3)$ by polar coordinates and the
rapid decay of the test function. Passing to the limit
proves the identity for the stated data.
\end{proof}

\begin{lemma}[From $a$ to $U$]\label{lemma:a to U}
Let $a\in\mathcal D'(\mathbb S^2;\mathbb C^3)$ satisfy
$a(-\omega)=\overline{a(\omega)}$. Define
\[
U(\xi)=\frac{\sqrt\pi}{2}
\left\langle a(\omega),e^{-(\omega\cdot\xi)^2/4}\right\rangle
+i\left\langle a(\omega),D(\omega\cdot\xi/2)\right\rangle,
\quad D(z)=e^{-z^2}\int_0^ze^{t^2}\,dt.
\]
Then $U\in C^\infty(\mathbb R^3;\mathbb R^3)$ and
\begin{equation}\label{eq:OU}
\Delta U+\frac12\xi\cdot\nabla U+\frac12U=0.
\end{equation}
If $\omega\cdot a=0$ distributionally, then
$\operatorname{div}U=0$.
\end{lemma}
\begin{proof}
The kernels depend smoothly on $(\xi,\omega)$, so the distributional pairings therefore
depend smoothly on $\xi$, with derivatives obtained by
differentiating the kernels. That $U$ is real follows by conjugation
and the substitution $\omega\mapsto-\omega$.
Put $s=\omega\cdot\xi$ and let
$K_e(s)=(\sqrt\pi/2)e^{-s^2/4}$ for the even kernel and $K_o(s)=iD(s/2)$ for the odd kernel.
Direct differentiation gives
\[
K_e''+\frac s2K_e'+\frac12K_e=0,\qquad
K_o''+ \frac{s}{2}K_o'+\frac{1}{2}K_o=0.
\]
Since $|\omega|=1$, these identities imply \eqref{eq:OU}
after pairing with $a$. Finally,
\[
\operatorname{div}U
=\left\langle\omega\cdot a,
K_e'(\omega\cdot\xi)
+K_o'(\omega\cdot\xi)\right\rangle=0
\]
under the stated tangency hypothesis.
\end{proof}

    Note that in fact {\textit{any}} distribution $a$ with $a(\omega)\cdot\omega =0$ gives rise to a solution of Equation \eqref{OU}, but its convergence to the data may be less well-understood. Finally, we prove the following.

\begin{proposition}
\label{prop:heat-reconstruction}
Let $v\in L^2(\mathbb S^2;\mathbb R^3)$ and suppose
$u_0(x)=|x|^{-1}v(x/|x|)$ is divergence free in
$\mathcal D'(\mathbb R^3)$. Then $U=e^\Delta u_0$ is real,
smooth, divergence free, and satisfies \eqref{OU}.
It equals the profile obtained from $a=Tv$ in
Lemmas \ref{lemma:v to a} and \ref{lemma:a to U}.
Moreover,
\begin{equation}\label{BLOWDOWN2}
U_R(x):=RU(Rx)=e^{R^{-2}\Delta}(u_0(x))\longrightarrow u_0(x)
\end{equation}
in $L^2_{\mathrm{loc}}(\mathbb R^3)$ and in
$\mathcal S'(\mathbb R^3)$, and pointwise at every Lebesgue point of
$u_0$. If $v$ is continuous, this gives pointwise convergence at every
$x\ne0$. The data uniquely determines every tempered solution
of \eqref{OU} with this distributional boundary value, by
Lemma \ref{lem:heat-converse}. If in addition $U$ is a
stationary solution of the Euler equations, then $U$ is homothetic.
\end{proposition}
\begin{proof}
Use of polar coordinates gives
\[
\int_{B_R}|u_0|^2=R\|v\|_{L^2(\mathbb S^2)}^2.
\]
Thus $u_0$ is locally square integrable and a tempered distribution.
Convolution with the heat kernel is smooth and
commutes with distributional divergence. Homogeneity and
a change of variables give
\[
e^{t\Delta}u_0(x)=t^{-1/2}U(x/\sqrt t).
\]
Differentiating this solution to the heat equation at $t=1$ proves
\eqref{OU}, and the same identity gives \eqref{BLOWDOWN}.
For completeness, we give a local approximate-identity argument
that also applies at the origin. On a fixed ball $B_M$, split
$u_0$ by a smooth cutoff equal to one on $B_{2M}$.
The cutoff part is in $L^2$, so convolution with the heat kernel tends
to it in $L^2$ and at its Lebesgue points. For the remainder,
$|x-y|\geq|y|/2$ when $x\in B_M$, $|y|\geq2M$, so we have
$$ C\int_{|y|\geq 2M} t^{-1/2}e^{-|x-y|^2/(4t)}u_0(y)dy \leq C_{M,v}t^{-1/2}e^{-M^2/(4t)}.$$
The latter term tends to zero when $t\to 0$. Convergence in $\mathcal S'$ follows
from the Schwartz approximate-identity property.
It remains to prove the distributional angular identity. In one dimension
\[
e^\Delta\delta(s)=\frac1{2\sqrt\pi}e^{-s^2/4},\qquad
e^\Delta\operatorname{PV}(1/s)=D(s/2).
\]
For the second identity, we first observe that the convolution $F= e^\Delta \operatorname{PV}(1/s)$ is odd and next that
direct differentiation gives $F'+sF/2=1/2$. Finally, we know that the unique
solution of this ODE with $F(0)=0$ is $D(s/2)$. The same identities
hold for kernels depending only on $s=\omega\cdot x$ in
$\mathbb R^3$. Apply convolution with the heat kernel to the
distributional identity in Lemma \ref{lemma:v to a}.
It commutes with the angular pairing after testing in $x$,
and yields precisely the kernels in Lemma \ref{lemma:a to U}.
This also proves the divergence free property of the angular
construction directly from the data. Uniqueness follows from
Lemma \ref{lem:heat-converse}.
Finally, solving the stationary Euler equations together with $\mathcal A U=0$
implies that $U$ is a homothetic profile.
\end{proof}

\begin{lemma}\label{lem:heat-converse}
Let $U\in\mathcal S'(\mathbb R^d)$ satisfy
$\Delta U+\tfrac12(U+x\cdot\nabla U)=0$.
If $RU(R\,\cdot)\to u_0$ in $\mathcal S'$ as $R\to\infty$,
then $U=e^\Delta u_0$.
\end{lemma}
\begin{proof}
Set $w(t,x)=t^{-1/2}U(x/\sqrt t)$. Then
$\partial_t w=\Delta w$ and $w(t)\to u_0$ in $\mathcal S'$.
In Fourier space,
\[
 \partial_t\bigl(e^{t|\xi|^2}\widehat w(t,\xi)\bigr)=0
 \quad\hbox{in }\mathcal D'(\mathbb R^d).
\]
The multiplier is used only against compactly supported frequency
test functions. Taking $t\downarrow0$ shows that this function that is constant in time must be precisely $\widehat u_0$, so $\widehat w(t)=e^{-t|\xi|^2}\widehat u_0$.
Letting $t=1$ proves the claim. 
\end{proof}
By the lemmas above, every homothetic profile satisfying
the assumption in Equation \eqref{WEAKDECAY} equals $e^\Delta u_0$, where $u_0$ is its
distributional blow-down limit. For angular $L^2$ data in dimension
three, Proposition \ref{prop:heat-reconstruction} therefore supplies
the almost-everywhere convergence to $u_0$ required in the proof of Proposition
\ref{NSTOEULER}.

\begin{remark}
	Let us point out that the boundary condition in \cite{JS1} and \cite{JS2} given by
	\begin{equation*}
		| U(x) - u_0(x) | = o(|x|^{-1})
		\qquad \text{ as } |x| \to + \infty 
	\end{equation*}
	implies that $u_0$ is the blow-down limit of $U$. 
\end{remark}

\section{Properties of Homothetic Solutions}

\subsection{Regularity of the Data}
In this subsection, we shall prove Theorem \ref{LIOU}. 
For $1\leq q\leq p$, our convention is that
$f\in\mathcal M_q^p(\mathbb R^3)$ if and only if there is
$C<\infty$ such that
\[
\int_{B_r(x)}|f|^q\leq Cr^{3-3q/p}
\quad\hbox{for every }x\in\mathbb R^3,\ r>0.
\]
The following result states explicitly the convergence used
to pass to the Euler equation. For homothetic
profiles with angular $L^2$ data, that convergence is supplied
by Proposition \ref{prop:heat-reconstruction}.
\begin{proposition}\label{NSTOEULER}
Let $U$ be a homothetic profile on $\mathbb R^3$ and suppose
$U\in\mathcal M_q^3(\mathbb{R}^3)$ for some $2<q\leq3$.
Assume additionally that $U_R(x)=RU(Rx)$ converges almost
everywhere to a measurable vector field $u_0$ as $R\to\infty$.
Then $u_0\in\mathcal M_q^3(\mathbb{R}^3)$ is homogeneous of degree $-1$
and is a divergence-free weak stationary Euler field on
all of $\mathbb R^3$. The convergence also holds in
$L^2_{\mathrm{loc}}(\mathbb R^3)$ and in $\mathcal S'(\mathbb{R}^3)$.
\end{proposition}
\begin{proof}
The critical Morrey estimate is invariant under this scaling:
\[
\int_{B_r(x)}|U_R|^q
=R^{q-3}\int_{B_{Rr}(Rx)}|U|^q\leq Cr^{3-q}.
\]
Fatou's lemma gives the identical bound for $u_0$.
On each compact set $K$, H\"older's inequality gives
\[
\int_E|U_R|^2\leq C_K|E|^{1-2/q},\qquad E\subset K,
\]
uniformly in $R$. Since $q>2$, this proves uniform integrability of the family $|U_R|^2$. An alternate proof of the uniform integrability uses the Morrey condition and the theorem of De Le Vall\'{e}e Poussin (cf. \cite{A}). By the assumed pointwise convergence and Vitali's
theorem (see \cite{B}),
$U_R\to u_0$ strongly in $L^2(K)$. In particular,
\[
\begin{aligned}
&\|U_R\otimes U_R-u_0\otimes u_0\|_{L^1(K)}\\
&\quad\leq(\|U_R\|_{L^2(K)}+\|u_0\|_{L^2(K)})
\|U_R-u_0\|_{L^2(K)}\longrightarrow0.
\end{aligned}
\]
Since $U$ is homothetic, $U$ solves the stationary Euler equations,
so every $U_R$ solves the stationary Euler equations. For every
$\phi\in C_c^\infty(\mathbb R^3;\mathbb R^3)$ with
$\operatorname{div}\phi=0$, the preceding limit gives
$\int(u_0\otimes u_0):\nabla\phi=0$.
The divergence free property passes to the limit as well.
The Morrey estimate also gives $\int_{B_r}|U_R|\leq Cr^2$
uniformly in $R$ and the same holds for $u_0$. Splitting a
Schwartz test function into a compactly supported part and into dyadic annuli improves
the local convergence to convergence in $\mathcal S'$ by a standard argument.
Finally $U_R(\lambda x)=\lambda^{-1}U_{\lambda R}(x)$
for $\lambda>0$. Passing to the distributional limit proves
$u_0(\lambda x)=\lambda^{-1}u_0(x)$, thus the $(-1)$-homogeneity of $u_0$.
\end{proof}
\begin{remark}
Writing $u_0(r\omega)=r^{-1}v(\omega)$, its local
$L^q$ membership on an annulus $\{x: 1<|x|<2\}$ gives $v\in L^q(\mathbb S^2)$.
At $q=3$, the Morrey convention above gives
$\mathcal M_3^3=L^3$. Note that a nonzero homogeneous data of degree
$-1$ cannot lie in this space because its radial integral
contains $\int_0^\infty dr/r$. Nonzero homogeneous data
only makes sense when $q<3$.
\end{remark}

\begin{remark}
    Note that the choice of Morrey space we make is not quite arbitrary. For instance, Taylor and Kato prove in \cite{Ta} and respectively \cite{K} that the Morrey space $\mathcal{M}^3_q(\mathbb{R}^3)$ admits a small-data global well-posedness result for mild solutions whenever $q\in (1,3]$. Of course, any solution satisfying the boundary conditions of Jia and \v{S}ver\'{a}k trivially belongs in this class as well.
\end{remark}

\subsubsection{Three-Dimensions}
In this subsection we prove Theorem \ref{LIOU}, which is a conditional Liouville theorem for homothetic self-similar solutions. The main idea of the proof is to instead show that the data (a $(-1)$-homogeneous solution to the Euler equations) is trivial by sharpening a similar proof from \cite{Sh}. We split the proof of Theorem \ref{LIOU} into two propositions, each of which handles one itemized hypothesis from the statement of Theorem \ref{LIOU}. We begin with:
\begin{proposition}\label{prop:sh}
    Suppose $U$ is a homothetic forward self-similar solution to the three-dimensional Navier-Stokes equations lying in the Morrey space $\mathcal{M}^3_q(\mathbb{R}^3)$ for some $2<q\leq 3$. Suppose that the forward self-similar solution corresponds to initial data
    $$ u_0 = \frac{\vec{v}+f\vec{n}}{|x|}\in L^q(\mathbb{S}^2),$$
    where $\vec{n}$ is the unit normal to the sphere and $\vec{v}$ is a tangent vector field to the sphere. Suppose $\vec{v}\in W^{1,4/3}(\mathbb{S}^2)$ and $f\in L^4(\mathbb{S}^2)$.
    Then $f=0$ and the head pressure $|\vec{v}|^2+2p=0$ as well. If, in addition, the set $\{\omega\in \mathbb{S}^2 : \vec{v}(\omega)=0\}$ has measure zero or is a closed set with $C^2$ boundary, then $\vec{v}=u_0=U=0$.
\end{proposition}
\begin{proof}
We know by Proposition \ref{NSTOEULER} that $u_0$ is a stationary solution of the three-dimensional Euler equations. This reduces to a system of differential equations on the sphere:
\begin{equation}\label{EULERS2}\begin{cases}
(\nabla^\perp \cdot v)v^\perp + \nabla\left( \frac{1}{2}|v|^2 + p\right) =0\\
v\cdot\nabla f = |v|^2 + f^2 + 2p\\
f + \nabla\cdot v=0.
\end{cases}\end{equation}
Here we have abused notation by writing $v$ instead of $\vec{v}$.
We remark that the first equation in the system above can be written as
$$v\cdot\nabla v +\nabla p =0,$$
therefore we can always recover the pressure from 
$$-\Delta p = \textrm{div}^2(v\otimes v) + |v|^2+2p.$$
Since $v\in W^{1,4/3}\subset L^4$, it follows from standard elliptic regularity that $p\in L^2$. Let 
$h:= 2p + |v|^2 \in L^2$. Then we have that $h,f$ satisfy the following transport equations:
$$v\cdot\nabla h=0$$
and
$$v\cdot \nabla f = h+f^2.$$
Since $v\in L^4, f\in L^4$, and $h \in L^2$, we have that $vfh \in L^1$, and in particular, $vfh$ is a distribution  (i.e. we can take its divergence).
Since $v\in W^{1,4/3}$, we can say that the following integrals are justified
$$0 = \int_{\mathbb{S}^2} \textrm{div }( vfh) = \int_{\mathbb{S}^2} (\nabla\cdot v)fh + \int_{\mathbb{S}^2} h(v \cdot\nabla f) + \int_{\mathbb{S}^2} f(v\cdot\nabla h) =\int_{\mathbb{S}^2} h^2.$$
From here we conclude that $h\equiv 0$, which, along with $v\in W^{1,4/3}$, implies that the distribution $(\nabla^\perp\cdot v)v^\perp=0$.   Now since $v\in L^4$ and $f^3\in L^{4/3}$, we get $vf^3 \in L^1$ and the following integrals are justified:
$$0 = \int_{\mathbb{S}^2} \textrm{div}(vf^3) = 2\int_{\mathbb{S}^2} f^4.  $$
We conclude that $f\equiv 0$, which implies that $\nabla \cdot v=0$. This concludes the proof that $u_0$ must be tangential to the sphere, divergence free, and with zero head pressure whenever $u_0$ lies in this regularity class. If in addition, $v$ has a Lebesgue null zero set, then we have to conclude $\nabla^\perp \cdot v=0$ a.e., so $v$ is harmonic and therefore zero on $\mathbb{S}^2$. Otherwise, if the boundary of the (closed) null set of $v$ is $C^2$, then whenever $v$ is not zero, it is harmonic (on a $C^2$ open subset of $\mathbb{S}^2$). Therefore, $v=0$ necessarily on that set as well.
\end{proof}
\begin{remark}
    In the proof above, there remains a pathological case: that the data $u_0 = \vec{v}/|x|$ is a distributionally divergence-free vector field that has $\vec{v}=0$ on a rough non-null set $Z$ and that $\textrm{curl }_{\mathbb{S}^2} \vec{v}=0$ in the distributional sense on the complement of $Z$ in $\mathbb{S}^2$. In the context of our Liouville theorem, this is an issue of the method of proof rather than a real obstruction to a rigidity result. In our forthcoming work, we bypass any such obstacle and prove rigidity of homothetic solutions for spherical data $u_0\in L^q(\mathbb{S}^2)$ for any $q>2$, which certainly subsumes the case discussed above regardless of the behavior of the support. 
\end{remark}

\begin{remark}
Instead of assumptions on the null set of $\vec{v}$, we may assume continuity of $\vec{v}$ to conclude the conditional Liouville theorem.
For instance, if $v\in W^{1,q}$ with $q>2$, then $v$ is continuous. Then we get
$(\nabla^\perp \cdot v)|v^\perp|^2=0$ a.e. on $\mathbb{S}^2$. Since $v$ is continuous, we conclude that the $L^q$ function $\nabla^\perp \cdot v=0$ a.e. on the open set $\{x\in \mathbb{S}^2 : |v(x)|\neq 0\}$. In addition, $\nabla^\perp \cdot v=0$ trivially on the interior of the zero set $\{x\in \mathbb{S}^2 : |v(x)|= 0\}$. We conclude that $\nabla^\perp\cdot v=0$ a.e. on $\mathbb{S}^2$. Combined with the knowledge that $\nabla\cdot v=0$ on $\mathbb{S}^2$, we know that $\Delta v=0$ a.e. on $\mathbb{S}^2$. This implies that $v=0$.
\end{remark}
We now prove the other part of Theorem \ref{LIOU}:
\begin{proposition}
   Suppose $U$ is a homothetic forward self-similar solution to the three-dimensional Navier-Stokes equations lying in the Morrey space $\mathcal{M}^3_q(\mathbb{R}^3)$ for some $2<q\leq 3$. Suppose that the forward self-similar solution corresponds to initial data
    $$ u_0 = \frac{\vec{v}+f\vec{n}}{|x|}\in L^q(\mathbb{S}^2),$$
    where $\vec{n}$ is the unit normal to the sphere and $\vec{v}$ is a tangent vector field to the sphere. Suppose
     $\vec{v},f \in C^\alpha(\mathbb{S}^2)$, for some $\alpha>1/2$.
    Then $f=\vec{v}=u_0=U=0$.
\end{proposition}
\begin{proof}
Note that since $\alpha>1/2$, the product $v\cdot \nabla q$ makes sense as a distribution in $B^{\alpha-1}_{\infty,\infty}$ for any $q\in C^\alpha(\mathbb{S}^2)$. We let  $|v|^2+f^2+2p := H$.

For any vector field or function $q$ on the sphere, we consider its mollification $q_\epsilon$. Following the idea from \cite{CET}, we have the following commutator estimate for any $q\in C^\alpha$. 
$$\|(v_\epsilon \cdot\nabla q_\epsilon) - (v\cdot\nabla q)_\epsilon\|_{L^\infty} \leq \epsilon^{2\alpha-1}\|q\|_{C^\alpha}\|v\|_{C^\alpha}.$$
Therefore, for every $\alpha>1/2$, we have:
$$ v_\epsilon\cdot\nabla f_\epsilon = H_\epsilon  +R_{1}(\epsilon),$$
and
$$ v_\epsilon\cdot\nabla h_\epsilon = R_{2}(\epsilon),$$
where $R_i(\epsilon)\to0$ as $\epsilon\to0$. Performing an integration by parts as in the proof of the nonexistence of $(-1)$-homogeneous Euler solutions in the $C^1$ category (see \cite{Sh} or consider $\textrm{div }(f_\epsilon v_\epsilon h_\epsilon)$) and taking $\epsilon\to 0$ implies that $f=H=0$ and moreover:
$$(\textrm{curl } v)v^\perp =0,$$
in the distributional sense (note that the product makes sense as a distribution in $B^{\alpha-1}_{\infty,\infty}$). Since $v$ is continuous, the set $U$ that is the complement in $\mathbb{S}^2$ of the zero set of $v$ is an open set. What we get from the above equality is that $\textrm{curl } v =0$ on $U$ in the distributional sense. Since $\textrm{div } v=0$ on $\mathbb{S}^2$, we conclude that $v$ is a distributionally harmonic vector field on $U$. In particular, $v$ is smooth on $U$ and is a classically harmonic vector field on $U$. However, we have already that $v=0$ on the complement of $U$. This yields that $v=0$ on $\mathbb{S}^2$. 
\end{proof}

\subsubsection{Two-Dimensions} 
The following result classifies homothetic profiles with continuous
angular Euler data and a distributional blow-down limit.
\begin{proposition}
Suppose $U$ is a homothetic forward self-similar solution to the
two-dimensional Navier--Stokes equations and
$RU(R\,\cdot)\to u_0=b(\omega)/r$ in $\mathcal S'(\mathbb R^2)$,
where $b\in C(\mathbb S^1;\mathbb R^2)$.
Suppose $\operatorname{div}u_0=0$ in $\mathcal D'(\mathbb R^2)$
and that $u_0$ solves the stationary Euler equations distributionally on
$\mathbb R^2\setminus\{0\}$. Then we necessarily have
    $$b(\omega) = A\omega^\perp,$$
    where $\omega = (\cos(\theta),\sin(\theta))$,
    in polar coordinates. In particular, $U$ is necessarily the Oseen vortex:
    $$U(r,\omega) = A\frac{1-e^{-\frac{r^2}{4}}}{r} \omega^\perp.$$
\end{proposition}
\begin{proof}
Write $b=f e_r+g e_\theta$. Incompressibility gives $g=c$ constant,
and incompressibility across the origin requires
$\int_0^{2\pi}f\,d\theta=0$. The momentum equation
has zero tangential component and radial component
$r^{-3}(cf'-f^2-c^2)+\partial_rP_0=0$. Hence $P_0$ is radial and,
for a real constant $K$,
\[
 cf'=f^2+c^2-K.
\]
If $c\ne0$, continuity gives $f\in C^1$; multiplying by $f'$
and integrating over one period gives
$c\int_0^{2\pi}(f')^2\,d\theta=0$, so $f$ is constant.
If $c=0$, the identity $f^2=K$ and continuity on the circle again
make $f$ constant. Its zero mean forces $f=0$.
Thus $u_0=Ae_\theta/r$ with $A=c$.

Lemma \ref{lem:heat-converse} gives $U=e^\Delta u_0$.
Since $\operatorname{curl}u_0=2\pi A\delta_0$,
$\operatorname{curl}U=(A/2)e^{-r^2/4}$.
The heat convolution is smooth, divergence free, and rotation
invariant, so its radial component vanishes. Integrating
$(rU_\theta)'=(A/2)r e^{-r^2/4}$ from zero gives
\[
 U(r,\omega)=A\frac{1-e^{-r^2/4}}r\,\omega^\perp.
\]
\end{proof}
\begin{remark}
If $u_0(r,\omega) = \frac{\omega^\perp}{r}$ exhibits non-uniqueness, then the other forward self-similar solution cannot be homothetic.
\end{remark}

\begin{remark}
    Note that the case of $(-1)$-homogeneous steady states of Euler on $\mathbb{R}^2$ corresponds to the point vortex solutions, so these cannot correspond to classical distributional solutions in the whole domain $\mathbb{R}^2$.
\end{remark}

\begin{remark}
    The $(-1)$-homogeneous shear flow $$\overline{U}_{1,\beta}(\xi_1,\xi_2)=\left( \frac{\sqrt{\pi}}{2} e^{-\frac{\xi_2^2}{4}} \text{erfi}\left(\frac{\xi_2}{2}\right),0\right)$$ with boundary data $(\mathrm{P.V.}(1/y),0)$, is not ruled out by this theorem since the angular data is far from being in $L^2(\mathbb{S}^1)$.
\end{remark}

\subsection{Regularity to Decay Relationship} This subsection makes the relationship between the regularity of the spherical data (in Fourier space or in physical space) and the decay of the resulting profile a bit clearer. 
Consider a profile constructed from an angular distribution
$a$ by Lemma \ref{lemma:a to U}. Write
\begin{equation}\label{ATOU}
\begin{aligned}
U_{\mathrm{even}}(x)&=\frac{\sqrt\pi}{2}
\left\langle a(\omega),e^{-(\omega\cdot x)^2/4}\right\rangle_\omega,\\
U_{\mathrm{odd}}(x)&=i\langle a(\omega),D(\omega\cdot x/2)\rangle_\omega,
\qquad U=U_{\mathrm{even}}+U_{\mathrm{odd}},
\end{aligned}
\end{equation}
for the even and odd parts of the profile, respectively. 
We have the following elementary lemma:
\begin{lemma}\label{LAPLACE}
For every unit vector $\eta$ and every $s>0$,
\[
\int_{\mathbb S^2}e^{-s^2(\omega\cdot\eta)^2/4}\,d\omega
=2\pi\int_{-1}^1e^{-s^2t^2/4}\,dt
\leq\frac{4\pi^{3/2}}s.
\]
\end{lemma}
\begin{proof}
Rotate $\eta$ to the point $(0,0,1)$, so we get an integral over a great circle. Then if we parameterize the great circle appropriately and let $t=\cos\theta$, we get the first equality. 
The inequality follows by extending the last integral to
$\mathbb R$ and evaluating its Gaussian integral.
\end{proof}
The following proposition is further evidence for our heuristic that the regularity of the spherical data is directly related to the decay of the profile. It shows that bounded Fourier spherical data guarantees an isotropic profile.
\begin{proposition}
\label{SUPPORTODECAY}
Suppose $a\in L^\infty(\mathbb S^2;\mathbb C^3)$ in
\eqref{ATOU}. Then
\[
\sup_{\eta\in\mathbb S^2}|U_{\mathrm{even}}(s\eta)|
\leq\frac{2\pi^2}s\|a\|_{L^\infty},\qquad s>0.
\]
In particular, if $a$ is real and even, the resulting
profile has isotropic $O(|x|^{-1})$ decay. If it is also
tangent to the sphere, the profile is divergence free. It is homothetic
provided it additionally satisfies the stationary Euler equations.
\end{proposition}
\begin{proof}
Take the absolute value in the even integral and apply
Lemma \ref{LAPLACE}. For real even $a$ the odd integral
vanishes. Tangency implies the divergence free condition by
Lemma \ref{lemma:a to U}. That lemma also gives
$\mathcal A U=0$, so solving the stationary Euler equations implies that $U$ is a homothetic profile.
\end{proof}

\section{Examples of Homothetic Solutions}
\label{sec:examples}

\subsection{Atomic Measures}

Every pair of Dirac masses supported at antipodal points $\pm \omega_0\in \mathbb{S}^2$ corresponds to what we term an Okamoto shear flow.

Let $\omega_0 \in S^2$ and let $\omega_1\in S^2$ be orthogonal to $\omega_0$ (in particular $\omega_1$ lies on the great circle $G_{\omega_0}$). Consider the distribution 
$$a(\omega) = (\delta_{\omega_0}(\omega) + \delta_{-\omega_0}(\omega))\omega_1.$$
Evidently, this distribution satisfies the symmetry and divergence-free conditions we require on $a(\omega)$, in particular $a(-\omega)=\overline{a(\omega)}$ and $\omega\cdot a(\omega)=0$. Also, it is clear that this is an even distribution, i.e. $a(-\omega) = a(\omega)$. Therefore, the corresponding solution to the Ornstein-Uhlenbeck equation is, according to Equation \eqref{ATOU} and up to constant prefactor,
$$U(r,\omega) = \int_{S^2} e^{-\frac{(\omega\cdot\Omega)^2 r^2}{4}} a(\Omega)d\Omega=\int_{S^2} e^{-\frac{(\omega\cdot\Omega)^2 r^2}{4}} \left((\delta_{\omega_0}(\Omega) + \delta_{-\omega_0}(\Omega))\omega_1\right)d\Omega .$$
Therefore, up to an unimportant constant prefactor:
$$U(r,\omega) = \omega_1 e^{-\frac{(\omega\cdot\omega_0)^2 r^2}{4}}.$$
To illustrate this example more clearly, let $\omega_0 = (0,1,0)$ and let, for instance $\omega_1= (1,0,0)$. Then we can say:
$$U(\xi_1,\xi_2,\xi_3) = \left( e^{-\xi_2^2/4},0,0\right).$$
This solution corresponds to a forward self-similar solution of Navier-Stokes from vortex sheet initial data $(\delta_0(y),0,0)$, which lies in $BMO^{-1}$, and was described by Okamoto in \cite{O}.

The other one-dimensional shear flow described by Okamoto comes from the distribution 
$$a(\omega) = -i\left( \delta_{\omega_0}(\omega) - \delta_{-\omega_0}(\omega)\right)\omega_1.$$
Indeed, in this case the distribution satisfies the symmetry and divergence-free conditions and is odd, so we only see the odd part:
$$U(r,\omega) = \int_{S^2} D\left(\frac{r(\omega \cdot\Omega)}{2}\right)\left( \delta_{\omega_0}(\Omega) - \delta_{-\omega_0}(\Omega)\right)\omega_1d\Omega.$$
Therefore, up to an unimportant constant prefactor:
$$U(r,\omega) = \omega_1 D\left(\frac{r(\omega\cdot\omega_0)}{2}\right).$$
To illustrate this example more clearly, let $\omega_0 = (0,1,0)$ and $\omega_1=(1,0,0)$, then we can say:
$$U(\xi_1,\xi_2,\xi_3) = \left( D(\xi_2/2), 0,0\right),$$
which is the other shear flow found by Okamoto. We also recognize this as the shear flow for which we shall later prove instability as promised in Theorem \ref{UNSTABLE} of the introduction.

Formally, an Okamoto shear flow is of the form:
$$S(\xi) =\omega_1 F(\xi \cdot \omega_0),$$
where $\omega_1\perp \omega_0$ and $F$ is either the Gaussian or Dawson function. As done with the Mikado flow building blocks of convex integration, it is clear that the sum of any two Okamoto shear flows $S(\xi)$ and $S'(\xi)$ with direction $\omega_1'$ and decay in $\omega_0'$ direction remains a homothetic self-similar solution if and only if 
$$\omega_1 \cdot \omega_0' =0 \textrm{ and } \omega_1'\cdot\omega_0 =0,$$
which, in $\mathbb{R}^3$, already forces at least one pair of directions to be parallel.

\subsection{Measures Supported on Curves}
Now we recover another example of Okamoto's. Let $\omega_0\in S^2$ be arbitrary, and consider the great circle $G_{\omega_0}$ of vectors in $S^2$ orthogonal to $\omega_0$.
Consider the following vector valued measure supported on the great circle $G_{\omega_0}$
$$a(\omega) = \delta_0(\omega\cdot\omega_0)\omega_0.$$
This distribution is even, so we get
$$U(r,\omega) =\int_{S^2} e^{-\frac{(\omega\cdot\Omega)^2 r^2}{4}} \delta_0(\Omega\cdot\omega_0)\omega_0d\Omega. $$
For illustrative purposes, let $\omega_0 = (1,0,0)$.
The integral above now becomes an integral over the great circle $\gamma :=G_{\omega_0}$, which we may parametrize by $\gamma(t)= (0, \cos(t), \sin(t))$. We get:
$$U(r,\omega) = (1,0,0)\int_0^{2\pi}e^{-\frac{r^2}{4}(\omega_2\cos(t)+\omega_3\sin(t))^2}dt. $$
However we note that up to unimportant prefactors:
$$
\int_{0}^{2\pi} \exp\left(-\frac{r^{2}}{4}(\omega_{2}\cos t+\omega_{3}\sin t)^{2}\right)dt
=$$ $$=
\exp\left(-\frac{r^{2}}{8}(\omega_{2}^{2}+\omega_{3}^{2})\right)
I_{0}\left(\frac{r^{2}}{8}(\omega_{2}^{2}+\omega_{3}^{2})\right),
$$
where $I_0$ denotes the modified Bessel function of the first kind.

In cylindrical coordinates with axis parallel to $\omega_0=(1,0,0)$, we have
$$U(x_1,\rho,\theta) = \left( e^{-\rho^2/8}I_0(\rho^2/8),0,0\right),$$
which is an axisymmetric shear flow identified by Okamoto in \cite{O}.

Our next particular solution shows that if $a(\omega)$ has the same support but a different vector field along the support, the corresponding homothetic solution can be quite different. Indeed, let $\omega_0\in S^2$ be arbitrary, and consider the great circle $G_{\omega_0}$ of vectors in $S^2$ orthogonal to $\omega_0$.
Consider the distribution supported on the great circle $G_{\omega_0}$.
$$a(\omega) = -i\delta_0(\omega\cdot\omega_0)(\omega\times \omega_0).$$
This is an odd distribution that respects the symmetry condition to recover a real-valued vector field. In particular, because $a$ is odd we have
$$U(r,\omega) = \int_{S^2} D\left(\frac{r(\omega \cdot\Omega)}{2}\right)\left(\delta_0(\Omega\cdot\omega_0)(\Omega\times \omega_0)\right)d\Omega.$$
Once again, this is simply an integral over the circle $G_{\omega_0}$, which we parametrize by $(0,\cos(t),\sin(t))$. The resulting integral is
$$U(r,\omega) = \int_0^{2\pi} D\left(\frac{r}{2}(\omega_2\cos(t)+\omega_3\sin(t))\right)(0,\sin(t),-\cos(t))dt.$$
We can evaluate the integral and get, up to unimportant prefactors:
$$U(\xi_1,\xi_2,\xi_3) = \frac{1-e^{-(\xi_2^2+\xi_3^2)/4}}{\xi_2^2+\xi_3^2}\left(0, -\xi_3, \xi_2\right).$$
In cylindrical coordinates, we see that 
$$U(x_1,\rho,\theta) = \frac{1-e^{-\rho^2/4}}{\rho}e_\theta,$$
which is a pure swirl vector field. Essentially, this flow is the Oseen vortex lifted to three dimensions. We leave open the question whether another vector field distribution on the great circle can lead to a homothetic self-similar solution. 

The Okamoto shear flow can be modified in the following way. Instead of having a distribution with support on the entire great circle, we allow the distribution to be supported on antipodal partial arcs of the great circle. If one expands the analytic integrand as a series and integrates term-by-term, one can find a series representation for the solutions in this case. 

First we consider the distribution
$$a(\omega) = \delta_0(\omega\cdot\omega_0)\omega_0(\chi_A(\omega) + \chi_{-A}(\omega)),$$
where $\omega_0\in S^2$ is fixed and $A$ is some fixed subarc of the great circle $G_{\omega_0}$ that is disjoint from $-A$. This remains an even, real-valued distribution and has the symmetry property $a(\omega)= \overline{a(-\omega)}$. Since the distribution is even, we get
$$U(r,\omega) =\int_{S^2} e^{-\frac{(\omega\cdot\Omega)^2 r^2}{4}}\delta_0(\Omega\cdot\omega_0)\omega_0(\chi_A(\Omega) + \chi_{-A}(\Omega)) d\Omega. $$
For illustrative purposes, let $\omega_0 = (1,0,0)$.
The integral above becomes an integral over the subarcs of the circle $\gamma :=G_{\omega_0}$ corresponding to $A$ and $-A$, and we may parametrize the curve $\gamma$ by $\gamma(t)= (0, \cos(t), \sin(t))$. Suppose the subarcs are exactly: $[\theta_1,\theta_2]$ and $[\pi+\theta_1,\pi+\theta_2]$. We get:
$$U(r,\omega) = $$ $$=(1,0,0)\bigg(\int_{\theta_1}^{\theta_2}e^{-\frac{r^2}{4}(\omega_2\cos(t)+\omega_3\sin(t))^2}dt +$$ $$+\int_{\theta_1+\pi}^{\theta_2+\pi}e^{-\frac{r^2}{4}(\omega_2\cos(t)+\omega_3\sin(t))^2}dt\bigg). $$
Now let $\phi$ be the angle such that
$$\cos(\phi) = \frac{\omega_2}{\sqrt{\omega_2^2+\omega_3^2}}, \quad \sin(\phi) =  \frac{\omega_3}{\sqrt{\omega_2^2+\omega_3^2}}.$$
In addition, let $\rho = \sqrt{x_2^2+x_3^2}$. Then:
$$U(r,\omega) = (1,0,0)\left(2e^{-\rho^2/8}\int_{\theta_1-\phi}^{\theta_2-\phi}e^{-\frac{\rho^2}{8}\cos(2t)}dt \right). $$
We use the series expansion for $e^{-z}$ at $z=0$:
$$ e^{-z} = \sum_{n\geq 0} \frac{(-1)^n}{n!} z^n.$$
Therefore, we may integrate term by term and get
$$U(r,\omega) = (2,0,0)e^{-\rho^2/8} \sum_{n\geq 0}\int_{\theta_1-\phi}^{\theta_2-\phi} \frac{(-1)^n}{n!} \left(\frac{\rho^2}{8} \cos(2t)\right)^ndt,$$
or
$$U(r,\omega) = (2,0,0)e^{-\rho^2/8} \sum_{n\geq 0}\frac{(-1)^n}{n!}\left(\frac{\rho^2}{8}\right)^n\int_{\theta_1-\phi}^{\theta_2-\phi}  \cos^n(2t)dt.$$

   We observe that the solution $U$ we construct is no longer axisymmetric, unless $\theta_2-\theta_1=\pi$ or $2\pi$, for which the integral above becomes independent of the angular variable and the solution reduces to a constant multiple of the previously discussed Okamoto shear flow. We also remark that since the profile above remains a shear flow, it is clearly a stationary solution of the three-dimensional Euler equations and therefore a homothetic solution. Let us compute the corresponding "data at infinity" by returning to the following general form of the profile:
   $$U(r,\omega) = (1,0,0)\left(2\int_{\theta_1-\phi}^{\theta_2-\phi}e^{-\frac{\rho^2}{4}\cos^2(t)}dt \right). $$
   Therefore, $RU(R r,\omega)$ is
   $$RU(Rr,\omega) = (1,0,0)\left(2R\int_{\theta_1-\phi}^{\theta_2-\phi}e^{-R^2\frac{\rho^2}{4}\cos^2(t)}dt \right). $$
As $R\to \infty$, we see that the integral only contributes something nontrivial in the limit if $\cos(t)=0$ somewhere in the interval $[\theta_1-\phi, \theta_2-\phi]$. Let $t_k$ be the roots of $\cos(t)$ in the interval, and let $N(\phi)$ be the number of roots in the interval. Then we have
for some arbitrary $\epsilon>0$:
$$R\int_{\theta_1-\phi}^{\theta_2-\phi} e^{\tfrac{-(R\rho)^2}{4}\cos^2(t)}dt = R\int_K e^{\tfrac{-(R\rho)^2}{4}\cos^2(t)}dt+$$ $$\sum_{t_k\in [\theta_1-\phi,\theta_2-\phi]} R\int_{t_k-\epsilon}^{t_k+\epsilon} e^{\tfrac{-(R\rho)^2}{4}\cos^2(t)}dt. $$
Here we denoted $K$ to be the (compact) complement of the union of intervals $(t_k-\epsilon,t_k+\epsilon)$ relative to the original interval depending on $\phi$. Since $\cos(t)$ does not vanish on $K$, as $R\to \infty$, the corresponding integral vanishes. For each term in the sum we have approximately:
$$\lim_{R\to \infty} \frac{R\rho}{\rho} \int_{t_k-\epsilon}^{t_k+\epsilon} e^{-\frac{(R\rho)^2}{4}\cos^2(t)}dt \approx\lim_{R\to \infty} \frac{R\rho/2}{\rho/2} \int_{t_k-\epsilon}^{t_k+\epsilon} e^{-\frac{(R\rho)^2}{4}(t-t_k)^2}dt= $$
$$= \frac{2}{\rho}\lim_{M\to \infty} M\int_{t_k-\epsilon}^{t_k+\epsilon} e^{-M^2(t-t_k)^2}dt = \frac{2}{\rho}\lim_{M\to \infty} \sqrt{\pi}\textrm{Erf}(M\epsilon) = \frac{2\sqrt{\pi}}{\rho}.$$
Therefore, up to an unimportant prefactor, we see that the data at infinity should be
$$u_0(\rho,\phi,x_1) = \frac{N(\phi)}{\rho}.$$
For illustrative purposes, if $\theta_1=\pi/2$ and $\theta_2=\pi$ as in our explicit example from before, then $N(\phi)$ is a step function:
$$N(\phi):= \begin{cases}
1, \textrm{ when } \phi\in [0,\pi/2] \cup [\pi, 3\pi/2]\\
0, \textrm{ when } \phi \in (\pi/2, \pi) \cup (3\pi/2, 2\pi).
\end{cases}$$
Note how the data corresponds to the support of the measure we began the analysis with.

The case just discussed, when the measure is supported on a curve, gets us a profile whose blow-down limit is essentially $1/\rho$, which fails to be in $L^2_{loc}$ (and therefore cannot be a weak solution to the Euler equations in $\mathbb{R}^3$). This does not contradict our Proposition \ref{NSTOEULER}, since that result required that the homothetic solution lie in the Morrey space $\mathcal{M}^3_q$ with $q>2$, while the profiles constructed above lie in $\mathcal{M}^3_q$ with $q<2$. Nevertheless, we know that small-data global well-posedness results have been proven in the class $\mathcal{M}^3_q$ for any $q>1$, so the question of uniqueness from the simple data $\beta/\rho$ for large $\beta$ remains open and interesting.

Modifying the support of the vector-valued measure for the lifted Oseen profile in a similar way leads to a different outcome. The profile remains a solution of the self-similar heat equation and remains divergence
free, but we are not guaranteed to obtain a profile that solves the stationary Euler equations (e.g. it does not have a shear flow structure). The question of what other vector valued measures would lead to homothetic solutions remains open and, in our view, is quite interesting.

\subsection{Instability}
Set $P(y)=D(y/2)$ where $D(z)$ is again the Dawson function. Set $\overline U_\beta(x)=\beta P(x_2)e_1$.
We first construct an unstable mode via Rayleigh's equation and an argument of Lin (from \cite{Lin}), then we obtain
approximate spectrum for the linearization of the three-dimensional Euler equations in velocity form on $L^2$ by constructing a wave packet along the lines suggested by the authors of \cite{HS}.
\begin{lemma}\label{lem:dawson-rayleigh}
Let $K=-P''/P$, with its value at zero defined by continuity, and let
$H=-d^2/dy^2-K$ on $H^2(\mathbb R)\subset L^2(\mathbb R)$.
Then $H$ has exactly one negative eigenvalue $-\alpha_0^2$.
For every $0<\alpha<\alpha_0$ there are $\gamma_\alpha>0$ and
$0\ne\phi_\alpha\in H^2(\mathbb R)$ such that
\[
(P-i\gamma_\alpha)(\phi_\alpha''-\alpha^2\phi_\alpha)
-P''\phi_\alpha=0.
\]
\end{lemma}
\begin{proof}
The proof of this lemma is a simple adaptation of the argument from \cite{Lin}.
\end{proof}
We now specify what we mean by the linearized Euler equations in velocity form. Take the unitary Fourier transform
in $(x_1,x_3)$, with frequencies $(k,\ell)$, and let
\[
\mathcal H_{k,\ell}=
\{V\in L^2(\mathbb R_y;\mathbb C^3):
ikV_1+\partial_yV_2+i\ell V_3=0\}.
\]
Let $\mathbb P_{k,\ell}$ be the orthogonal projection onto this space.
Define the bounded operators
\[L(k,\ell)V=-\mathbb P_{k,\ell}
\bigl(ikP V+P'V_2e_1\bigr)
\quad\hbox{on }\mathcal H_{k,\ell}.
\]
The direct integral of the operators $L(k,l)$ defines $L$ on
$L^2_\sigma(\mathbb R^3;\mathbb C^3)$, with domain
\[
D(L)=\left\{v\in L^2_\sigma(\mathbb R^3;\mathbb C^3):
\int_{\mathbb R^2}\|L(k,\ell)\widehat v(k,\ell,\cdot)\|_2^2
\,dk\,d\ell<\infty\right\}.
\]
This is a closed, densely defined realization of the three-dimensional  Euler equations linearized about the Okamoto shear flow: 
$-\mathbb P(P\partial_1v+P'v_2e_1)$. In fact, the domain above can easily be seen to be the maximal domain of the operator on $L^2$.
Define $L_\beta=\beta L$ on $D(L)$ for
$\beta\ne0$.
\begin{theorem}\label{thm:unstable}
For every $\beta\in\mathbb R\setminus\{0\}$, the velocity
linearization $L_\beta$ around $\overline U_\beta=\beta D(x_2/2)e_1$
has a positive real point in its approximate spectrum on $L^2_\sigma(\mathbb{R}^3,\mathbb{C}^3)$.
More precisely, for each $0<\alpha<\alpha_0$,
\[
|\beta|\alpha\gamma_\alpha\in\sigma_{\rm ap}(L_\beta),
\]
where $\gamma_\alpha$ is furnished by Lemma
\ref{lem:dawson-rayleigh}.
\end{theorem}
\begin{proof}
Fix a mode from Lemma \ref{lem:dawson-rayleigh} and write $\phi=\phi_\alpha$,
$c=i\gamma_\alpha$, and $\lambda=-i\alpha c>0$.
At $(k,\ell)=(\alpha,0)$ set
$V_\alpha=(\phi',-i\alpha\phi,0)$.
The Rayleigh equation implies
$L(\alpha,0)V_\alpha=\lambda V_\alpha$.
For completeness, the associated pressure is
$p=-(P-c)\phi'+P'\phi$; differentiation and the Rayleigh equation give
$p'=-\alpha^2(P-c)\phi$. These identities verify the two nonzero
components of the velocity eigenvalue equation before projection.
Near $(\alpha,0)$ set $V_{k,\ell}=(\phi',-ik\phi,0)$.
Its divergence on each fiber is exactly zero. The projections
$\mathbb P_{k,\ell}$ vary continuously in operator norm in this
neighborhood: after Fourier transform in $y$, their symbols are the
orthogonal projections perpendicular to $(k,\eta,\ell)$, and
$k^2+\eta^2+\ell^2\geq\alpha^2/4>0$ there. The symbols and their
$(k,\ell)$ derivatives are uniformly bounded in $\eta$.
It follows that
\[
\|(L(k,\ell)-\lambda)V_{k,\ell}\|_{L^2_y}\longrightarrow0
\quad\hbox{as }(k,\ell)\to(\alpha,0).
\]
Choose $a\in C_c^\infty(\mathbb R^2)$ with support in the unit ball of $\mathbb{R}^2$ and with $\|a\|_2=1$. Let
\[
a_n(k,\ell)=n\,a(n(k-\alpha),n\ell),\qquad
\widehat v_n=N_n^{-1}a_n(k,\ell)V_{k,\ell},
\]
where $N_n$ is the constant that makes $\|v_n\|_2=1$. This is the wave packet construction. Then
$N_n\to\|V_\alpha\|_2>0$. Each $v_n$ is divergence free and belongs
to $D(L)$ since its horizontal frequencies are bounded and
$\|L(k,\ell)\|\leq |k|\|P\|_\infty+\|P'\|_\infty$.
Plancherel's identity gives
\[
\|(L-\lambda)v_n\|_2^2
=N_n^{-2}\int_{\mathbb R^2}|a_n(k,\ell)|^2
\|(L(k,\ell)-\lambda)V_{k,\ell}\|_2^2\,dk\,d\ell
\longrightarrow0.
\]
This proves $\lambda\in\sigma_{\rm ap}(L)$. The analogous claim for $L_\beta$ follows easily.
\end{proof}

\begin{remark} 
We note that a similar analysis can be done for the other shear flows identified by Okamoto, they all appear to admit the existence of unstable parts of the spectrum that are approximate eigenvalues.
\end{remark}

We note that, for analytic semigroups, even an approximate unstable eigenvalue guarantees the existence of a nontrivial trajectory backwards in time $t\to -\infty$ (see \cite{Lu}), which would be sufficient for the non-uniqueness program, as we discussed in Section 2. Note that neither the linearization of the self-similar Navier-Stokes equation nor the linearization of the Euler equation gives rise to an analytic semigroup. In the example above, then, the primary barrier to proving non-uniqueness from the data $1/y$ appears to be making sense of the singular limit to the Navier-Stokes equations when the profile anisotropically decays. We plan to investigate this direction further. 

\section{Stability of the Oseen Vortex}
We study individual angular modes of the two-dimensional Euler
linearization in vorticity variables.
Let $\overline u(x)=u(r)e_\theta$ be a smooth real vector field on
$\mathbb R^2$, and write
\[
\Omega(r)=\frac{u(r)}r,\qquad
\overline\omega(r)=u'(r)+\frac{u(r)}r.
\]
Assume that $\Omega$ is bounded and
$\overline\omega'(r)=O(r^{-3})$ as $r\to\infty$.
Smoothness of the planar field implies
$\overline\omega'(r)=O(r)$ at zero, so $r\overline\omega'$ is bounded.
For $k\in\mathbb Z$ with $|k|\geq2$, set
$\mathcal X_k=L^2((0,\infty),r\,dr;\mathbb C)$ and define
\[
(L_k g)(r)=-ik\Omega(r)g(r)
+ik\frac{\overline\omega'(r)}r f(r).
\]
Here the streamfunction is
\[
f(r)=-\frac1{2|k|}\left[
\begin{aligned}
&r^{-|k|}\int_0^r s^{|k|+1}g(s)\,ds\\
&+r^{|k|}\int_r^\infty s^{1-|k|}g(s)\,ds
\end{aligned}\right].
\]
Both integrals exist for $g\in\mathcal X_k$ by the Cauchy-Schwarz inequality.
This fixes the stream function normalization and shows that $f$ satisfies
$f''+f'/r-k^2f/r^2=g$.
To see that $L_k$ is bounded on $\mathcal X_k$, set
$G(s)=e^sg(e^s)$ and $F(s)=e^{-s}f(e^s)$. Then
\[
G=T_kF,\qquad T_k=\partial_s^2+2\partial_s+1-k^2.
\]
The Fourier symbol of $T_k$ is $-\eta^2+2i\eta+1-k^2$, which has no real zero for
$|k|\geq2$, and its inverse maps $L^2(\mathbb R)$ boundedly into
$H^2(\mathbb R)$. The same decomposition into modes and logarithmic change of coordinates was employed in \cite{BC}. The Green formula above is the inverse that gives our stream function $f$ and equivalently $F$. Indeed,
writing $n=|k|$, it gives
\[
F(s)=-\frac1{2n}\left[
\int_{-\infty}^s e^{-(n+1)(s-t)}G(t)\,dt
+\int_s^\infty e^{-(n-1)(t-s)}G(t)\,dt\right].
\]
Both one-sided kernels are integrable, so $F\in L^2$; the equation
$T_kF=G$ and the Fourier multiplier then give $F\in H^2$.
Under the unitary map
$g\mapsto G$, the operator becomes
\[
G\longmapsto-ik\Omega(e^s)G
+ik e^s\overline\omega'(e^s)T_k^{-1}G.
\]
It is therefore a bounded operator, with domain all of $\mathcal X_k$.
\begin{theorem}
\label{thm:radial-modal}
Let $\overline u=u(r)e_\theta$ be a smooth real vector field on
$\mathbb R^2$, with bounded $\Omega=u/r$ and
$\overline\omega'=O(r^{-3})$ at infinity, where
$\overline\omega=u'+u/r$. Suppose either
$\overline\omega'(r)>0$ for every $r>0$, or
$\overline\omega'(r)<0$ for every $r>0$.
Then, for each integer $|k|\geq2$,
\[
\sigma_{\rm p}(L_k)\subset i\mathbb R.
\]
Here $\sigma_{\rm p}$ denotes point spectrum on $\mathcal X_k$.
\end{theorem}
\begin{proof}
Suppose $L_kg=\lambda g$, where $g\in\mathcal X_k$,
$\lambda=a+ib$, and $a\ne0$.
Since $F=T_k^{-1}G\in H^2(\mathbb R)\subset L^\infty(\mathbb R)$,
the eigenvalue equation gives
\[
g(r)=\frac{ik\overline\omega'(r)}{\lambda+ik\Omega(r)}F(\log r),
\qquad |g(r)|\leq \frac{|k|}{|a|}
\|F\|_\infty|\overline\omega'(r)|.
\]
In particular, $g(r)=O(r^{-3})$ at infinity. Writing $n=|k|\geq2$,
the two terms in the Green formula satisfy
\[
r^{-n}\int_0^r s^{n+1}g(s)\,ds=O(r^{-1}),\qquad
r^n\int_r^\infty s^{1-n}g(s)\,ds=O(r^{-1})
\quad(r\to\infty).
\]
For the first estimate split the integral at one; its tail is bounded
by a constant times $\int_1^r s^{n-2}\,ds$.
Hence $f(r)=O(r^{-1})$ at infinity. At zero the boundedness of $F$
gives $f(r)=O(r)$. Consequently $h(s)=f(e^s)$ belongs to
$L^2(\mathbb R)$.
The eigenvalue equation and the streamfunction equation imply
\[
h''-k^2h=Q_{\lambda,k}h,\qquad
Q_{\lambda,k}(s)
=\frac{ik e^s\overline\omega'(e^s)}
{\lambda+ik\Omega(e^s)}.
\]
This coefficient is bounded, since $|\lambda+ik\Omega|\geq|a|$.
Therefore $h''\in L^2$ distributionally and $h\in H^2(\mathbb R)$.
Integration by parts is now justified and gives
\[
\int_{\mathbb R}(|h'|^2+k^2|h|^2)\,ds
=-\int_{\mathbb R}Q_{\lambda,k}|h|^2\,ds.
\]
But
\[
\operatorname{Im}Q_{\lambda,k}(s)
=\frac{k a e^s\overline\omega'(e^s)}
{a^2+(b+k\Omega(e^s))^2}
\]
has a strict sign on $\mathbb R$. Taking imaginary parts forces
$h=0$, and hence $f=g=0$. Thus no eigenvalue has nonzero real part.
\end{proof}
For the radial mode $k=0$, set
$\mathcal X_0=L^2((0,\infty),r\,dr;\mathbb C)$ and define
$L_0=0$ on $\mathcal X_0$.
This agrees with the Euler linearization since a radial perturbation has
tangential induced velocity, so both vorticity transport terms vanish.
The modes $k=\pm1$ are outside the scope of our theorem.
\begin{corollary}\label{cor:oseen-modal}
For the Oseen vortex
\[
\overline u(r,\theta)=\frac{1-e^{-r^2/4}}r e_\theta,
\qquad \overline\omega(r)=\frac12e^{-r^2/4},
\]
the operators $L_k$ just defined satisfy
$\sigma_{\rm p}(L_k)\subset i\mathbb R$ for every integer $|k|\geq2$,
and $L_0=0$.
\end{corollary}
\begin{proof}
The vector field extends smoothly through the origin,
$\Omega(r)=(1-e^{-r^2/4})/r^2$ is bounded, and
$\overline\omega'(r)=-r e^{-r^2/4}/4<0$ for $r>0$.
The hypotheses of Theorem \ref{thm:radial-modal} hold.
\end{proof}
\begin{corollary}\label{cor:even-mode-stability}
Under the hypotheses of Theorem \ref{thm:radial-modal}, the
linearization of the Euler equations in vorticity form around the Oseen vortex on the space
\[
L^2_2(\mathbb R^2;\mathbb C)
 :=\{g\in L^2(\mathbb R^2;\mathbb C):g(-x)=g(x)\}
\]
generates a strongly continuous group and satisfies
$\sigma(L)\subset i\mathbb R$ and $s(L)=\omega_0(L)=0$.
In particular, this proves Theorem \ref{OSEENSTABLE}.
\end{corollary}
\begin{proof}
We use the angular decomposition into even modes and the
logarithmic change of coordinates from above. Put
\[
 q(s)=e^s\overline\omega'(e^s),\qquad
 A_k=-ik\Omega(e^s),\qquad B_k=ikqT_k^{-1}
 \quad(k\in2\mathbb Z\setminus\{0\}),
\]
and set $A_0=B_0=0$. Define $A=\bigoplus A_k$ on its maximal
multiplication domain
\[
 D(A)=\left\{(G_k)\in\bigoplus_{k\in2\mathbb Z}L^2(\mathbb R):
 \sum_k\|k\Omega(e^s)G_k\|_2^2<\infty\right\},
\]
and $L=A+B$ on $D(A)$, where $B=\bigoplus B_k$.
The operator $A$ is skew-adjoint. Since $q\in C_0(\mathbb R)$,
the compactness from the Sobolev embedding and a cutoff approximation make each $B_k$ compact.
Moreover, the symbol of $T_k$ gives
\[
 \|B_k\|\leq\frac{|k|\|q\|_\infty}{k^2-1}\longrightarrow0.
\]
Thus $B$ is compact on $L^2$, and the bounded perturbation theorem gives the
strongly continuous group generated by $L$.

For $\operatorname{Re}z\ne0$, compact perturbation of $A$ makes
$z-L$ Fredholm of index zero. A nonzero vector in its kernel would
have a nonzero angular component contradicting
Theorem \ref{thm:radial-modal}. Note also that the radial component satisfies $L_0=0$.
Hence $\sigma(L)\subset i\mathbb R$, and $0\in\sigma(L)$
because the radial subspace is in the kernel.

Fix $\varepsilon>0$. For all sufficiently large $|k|$,
$\|B_k\|\leq\varepsilon/2$, so the Neumann series gives
\[
 \sup_{\operatorname{Re}z\geq\varepsilon}
 \|(z-L_k)^{-1}\|\leq2/\varepsilon.
\]
Each of the remaining finitely many $L_k$ is a bounded operator and has
spectrum on the imaginary axis, so its resolvent is bounded on
that half-plane, using continuity on bounded sets and its
Neumann expansion at infinity. The direct-sum resolvent is therefore
uniformly bounded there. The Gearhart--Pr\"uss theorem \cite[Theorem V.1.11]{EN} gives
$\|e^{tL}\|\leq M_\varepsilon e^{\varepsilon t}$ for $t\geq0$.
Since $\varepsilon$ is arbitrary and $0\in\sigma(L)$,
$\omega_0(L)=s(L)=0$.
\end{proof}

\end{document}